\documentclass[a4paper,11pt]{amsart}
\usepackage{mathtools,amssymb,amsthm}
\usepackage[all]{xy}

\theoremstyle{plain}
\newtheorem{thm}{Theorem}[section]
\newtheorem{lem}[thm]{Lemma}
\newtheorem{prop}[thm]{Proposition}
\newtheorem{cor}[thm]{Corollary}

\theoremstyle{definition}
\newtheorem{defn}[thm]{Definition}

\newtheorem{rem}[thm]{Remark}
\newtheorem{ex}[thm]{Example}

\numberwithin{equation}{section}

\newcommand{\eps}{\varepsilon}
\newcommand{\abs}[1]{\lvert #1 \rvert}

\newcommand{\Z}{\mathbb{Z}}
\newcommand{\Q}{\mathbb{Q}}
\newcommand{\R}{\mathbb{R}}
\newcommand{\C}{\mathbb{C}}
\newcommand{\bk}{\mathbf{k}}
\newcommand{\bl}{\mathbf{l}}
\newcommand{\bs}{\mathbf{s}}
\newcommand{\bz}{\mathbf{z}}
\newcommand{\bX}{\mathbf{X}}
\newcommand{\bY}{\mathbf{Y}}
\newcommand{\calZ}{\mathcal{Z}}
\newcommand{\calY}{\mathcal{Y}}

\newcommand{\Li}{\mathrm{Li}}

\newcommand{\wt}[1]{\lvert #1\rvert}

\keywords{Kaneko-Tsumura multiple zeta function, multiple zeta values}
\subjclass[2010]{11M32}

\title[Multiple zeta of Kaneko-Tsumura type]
{Multiple zeta functions of Kaneko-Tsumura type and their values at positive integers}

\author{Shuji Yamamoto}
\address{Department of Mathematics, Faculty of Science and Technology, Keio University, 
3-14-1 Hiyoshi, Kohoku-ku, Yokohama, 223-8522, Japan}
\email{yamashu@math.keio.ac.jp}
\thanks{This work was supported in part by JSPS KAKENHI JP26247004, JP18H05233, JP18K03221, JP21K03185, JSPS Core-to-core program ``Foundation of a Global Research Cooperative Center in Mathematics focused on Number Theory and Geometry'', JSPS Joint Research Project with CNRS ``Zeta functions of several variables and applications'', and the KiPAS program 2013–2018 of the Faculty of Science and Technology at Keio University. }

\begin{document}

\begin{abstract}
Kaneko and Tsumura introduced a new kind of multiple zeta functions 
$\eta(k_1,\ldots,k_r;s_1,\ldots,s_r)$. 
This is an analytic function of complex variables $s_1,\ldots,s_r$, 
while $k_1,\ldots,k_r$ are non-positive integer parameters. 
In this paper, we first extend this function to an analytic function 
$\eta(s'_1,\ldots,s'_r;s_1,\ldots,s_r)$ of $2r$ complex variables. 
Then we investigate its special values at positive integers. 
In particular, we prove some linear relations among these $\eta$-values and 
the multiple zeta values $\zeta(k_1,\ldots,k_r)$ of Euler-Zagier type. 
\end{abstract}

\maketitle

\section{Introduction}
In \cite{KT}, M.~Kaneko and H.~Tsumura introduced and studied a new kind of multiple zeta functions 
\begin{equation}\label{eq:eta(r;1)}
\eta(k_1,\ldots,k_r;s)\coloneqq \frac{1}{\Gamma(s)}
\int_0^\infty \frac{\Li_{k_1,\ldots,k_r}(1-e^t)}{1-e^t}t^{s-1}dt, 
\end{equation}
which is a `twin sibling' of the Arakawa-Kaneko multiple zeta function \cite{AK}
\begin{equation}\label{eq:xi}
\xi(k_1,\ldots,k_r;s)\coloneqq \frac{1}{\Gamma(s)}
\int_0^\infty\frac{\Li_{k_1,\ldots,k_r}(1-e^{-t})}{e^t-1}t^{s-1}dt. 
\end{equation}
Here $k_1,\ldots,k_r$ are integers, $s$ is a complex variable 
and $\Li_{k_1,\ldots,k_r}$ denotes the multiple polylogarithm of one variable 
\[\Li_{k_1,\ldots,k_r}(z)\coloneqq\sum_{0<n_1<\cdots<n_r}
\frac{z^{n_r}}{n_1^{k_1}\cdots n_r^{k_r}}. \]
Among other things, when $r=1$, they proved the equality 
\begin{equation}\label{eq:eta(k;l)=eta(l;k)}
\eta(k;l)=\eta(l;k)
\end{equation}
for nonpositive integers $k,l$, and experimentally observed 
that the same equality holds even when $k$ and $l$ are positive integers. 

In \cite[\S 5]{KT}, Kaneko and Tsumura also considered 
a variant of \eqref{eq:eta(r;1)} with $r$ complex variables: 
\begin{multline}\label{eq:eta(r;r)}
\eta(k_1,\ldots,k_r;s_1,\ldots,s_r)\\
\coloneqq 
\frac{1}{\prod_{j=1}^r\Gamma(s_j)}
\idotsint_0^\infty
\frac{\Li_{k_1,\ldots,k_r}
(1-e^{\sum_{\nu=1}^r t_\nu},1-e^{\sum_{\nu=2}^r t_\nu},\ldots,1-e^{t_r})}
{\prod_{j=1}^r(1-e^{\sum_{\nu=j}^r t_\nu})}\\
\qquad\qquad\times\prod_{j=1}^r t_j^{s_j-1}dt_j, 
\end{multline}
where 
\begin{equation}\label{eq:Li}
\Li_{k_1,\ldots,k_r}(z_1,\ldots,z_r)
\coloneqq\sum_{0<n_1<\cdots<n_r}
\frac{z_1^{n_1}z_2^{n_2-n_1}\cdots z_r^{n_r-n_{r-1}}}
{n_1^{k_1}\cdots n_r^{k_r}} 
\end{equation}
is the multiple polylogarithm of $r$ variables. 
For certain technical reasons, their consideration on the function \eqref{eq:eta(r;r)} is 
limited to the case that $k_1,\ldots,k_r$ are nonpositive integers. 

\bigskip

In the present paper, we extend the function \eqref{eq:eta(r;r)} to a holomorphic function 
of $2r$ complex variables $\eta(s'_1,\ldots,s'_r;s_1,\ldots,s_r)$, which satisfies 
\begin{equation}\label{eq:symmetry}
\eta(s'_1,\ldots,s'_r;s_1,\ldots,s_r)
=\eta(s_1,\ldots,s_r;s'_1,\ldots,s'_r).
\end{equation}
When $r=1$, it also gives an extension of the function \eqref{eq:eta(r;1)}. 
In particular, we obtain a proof of the equality 
\eqref{eq:eta(k;l)=eta(l;k)} for arbitrary complex numbers $k,l$. 

The second and the main purpose of this paper is to study the special values of the function 
$\eta(s'_1,\ldots,s'_r;s_1,\ldots,s_r)$ at positive integers. 
We prove certain linear relations among these values and the multiple zeta values 
\begin{equation}\label{eq:MZV}
\zeta(k_1,\ldots,k_r)\coloneqq
\sum_{0<n_1<\cdots<n_r}\frac{1}{n_1^{k_1}\cdots n_r^{k_r}} 
\end{equation}
for positive integers $k_1,\ldots,k_r$ with $k_r>1$. 
For example, as special cases of our result (Theorem \ref{thm:sum formula}), we can show that 
\[\eta(\underbrace{1,\ldots,1}_r;\underbrace{1,\ldots,1}_r)=\zeta(\underbrace{2,\ldots,2}_r)\]
and 
\[\eta(k;l)=\sum_{0<a_1\leq\cdots\leq a_k=b_l\geq\cdots\geq b_1>0}
\frac{1}{a_1\cdots a_k\,b_1\cdots b_l} \]
(the right hand side of the latter identity can be expressed as a finite sum of multiple zeta values). 

\bigskip 

The contents of this paper is as follows. 
In \S2, we define the function $\eta(s'_1,\ldots,s'_r;s_1,\ldots,s_r)$ and 
prove its analytic continuation to $\C^{2r}$ by the classical contour integral method. 
In \S3, basic formulas on its special values at positive integers are obtained. 
Some of them are used in \S4, 
where we discuss relations of $\eta$-values with the multiple zeta values. 
Finally, in the appendix A, we prove a formula which expresses 
the values $\eta(k_1,\ldots,k_r;l)$ of the function in \eqref{eq:eta(r;1)}, 
where $k_1,\ldots,k_r$ and $l$ are positive integers, in terms of the multiple zeta values. 

\section{Definition of $\eta(s'_1,\ldots,s'_r;s_1,\ldots,s_r)$}

Let $r$ be a positive integer. 
The definition \eqref{eq:Li} of the multiple polylogarithm is meaningful 
for arbitrary complex numbers $k_1,\ldots,k_r$ and 
complex numbers $z_1,\ldots,z_r$ of absolute values less than $1$. 
We begin with its analytic continuation. 
For a positive real number $\eps$, denote by 
$C_\eps$ the contour which goes from $+\infty$ to $\eps$ 
along the real line, goes round counterclockwise 
along the circle of radius $\eps$ about the origin, and then 
goes back to $+\infty$ along the real line: 
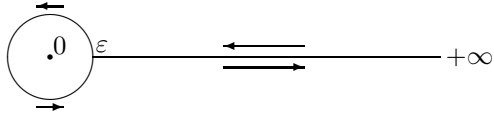
\[\begin{picture}(180,40)
\put(15,20){\circle{30}}
\put(31,20){\line(1,0){130}}
\put(15,20){\circle*{2}}
\put(16,21){$0$}
\put(32,22){$\eps$}
\put(163,17.5){$+\infty$}
\put(110,24){\vector(-1,0){30}}
\put( 20,39){\vector(-1,0){10}}
\put( 10, 1){\vector( 1,0){10}}
\put( 80,16){\vector( 1,0){30}}
\end{picture} \]

\begin{lem}
The multiple polylogarithm $\Li_{\bs}(\bz)$, 
where $\bs=(s_1,\ldots,s_r), \bz=(z_1,\ldots,z_r)\in\C^r$ and 
$\lvert z_i\rvert<1$, has the following integral expression: 
\begin{equation}\label{eq:Li_int}
\Li_{\bs}(\bz)=\prod_{j=1}^r\frac{\Gamma(1-s_j)}{2\pi ie^{\pi is_j}}
\int_{(C_\eps)^r} 
\prod_{j=1}^r\frac{z_ju_j^{s_j-1}du_j}{e^{u_j+\cdots+u_r}-z_j}. 
\end{equation}
Here we assume that $\eps>0$ is sufficiently small. 

By \eqref{eq:Li_int}, $\Li_{\bs}(\bz)$  is holomorphically 
continued to the region 
\[(\bs,\bz)\in \C^r\times(\C\setminus\R_{\geq 1})^r.\] 
\end{lem}
\begin{proof}
First we note that 
\begin{align*}
\prod_{j=1}^r \Gamma(s_j)\cdot\Li_{\bs}(\bz)
&=\prod_{j=1}^r \Gamma(s_j)\sum_{l_1,\ldots,l_r>0}
\frac{z_1^{l_1}z_2^{l_2}\cdots z_r^{l_r}}
{l_1^{s_1}(l_1+l_2)^{s_2}\cdots(l_1+\cdots+l_r)^{s_r}}\\
&=\sum_{l_1,\ldots,l_r>0}\idotsint_0^\infty 
\prod_{j=1}^r e^{-(l_1+\cdots+l_j)u_j}
z_j^{l_j}u_j^{s_j-1}du_j\\
&=\idotsint_0^\infty \prod_{j=1}^r 
\frac{z_ju_j^{s_j-1}du_j}{e^{u_j+\cdots+u_r}-z_j}, 
\end{align*}
that is, 
\begin{equation}\label{eq:Li_int'}
\Li_{\bs}(\bz)=\prod_{j=1}^r\frac{1}{\Gamma(s_j)}
\idotsint_0^\infty \prod_{j=1}^r 
\frac{z_ju_j^{s_j-1}du_j}{e^{u_j+\cdots+u_r}-z_j}. 
\end{equation}
This gives an analytic continuation to the region 
\[(\bs,\bz)\in \bigl\{s\in\C\bigm| \Re(s)>0\bigr\}^r
\times(\C\setminus\R_{\geq1})^r. \]
Moreover, for each $\bz\in(\C\setminus\R_{\geq 1})^r$, 
there exists a neighborhood $K$ of $\bz$ and $\eps_0>0$ 
such that $e^{u_j+\cdots+u_r}-z'_j\ne 0$ for $j=1,\ldots,r$ 
whenever $(z'_1,\ldots,z'_r)\in K$, $0<\eps<\eps_0$ 
and $u_1,\ldots,u_r\in C_\eps$. 
If this is the case, we have 
\begin{equation}\label{eq:contour}
\idotsint_0^\infty \prod_{j=1}^r 
\frac{z_ju_j^{s_j-1}du_j}{e^{u_j+\cdots+u_r}-z_j}
=\prod_{j=1}^r\frac{1}{e^{2\pi is_j}-1}
\int_{(C_\eps)^r} \prod_{j=1}^r 
\frac{z_ju_j^{s_j-1}du_j}{e^{u_j+\cdots+u_r}-z_j}. 
\end{equation}
The formula \eqref{eq:Li_int} is deduced from 
\eqref{eq:Li_int'} and \eqref{eq:contour} because 
\[\Gamma(s_j)\Gamma(1-s_j)=\frac{\pi}{\sin \pi s_j}
=\frac{2\pi i}{e^{\pi is_j}-e^{-\pi is_j}}
=\frac{2\pi i e^{\pi i s_j}}{e^{2\pi i s_j}-1}. \]
It is easy to see that \eqref{eq:Li_int} gives 
a meromorphic continuation to 
$\C^r\times(\C\setminus\R_{\geq 1})^r$. 
The possible poles $s_j=1,2,3,\ldots$, 
which comes from the factor $\Gamma(1-s_j)$, 
are removable, since we already know 
the holomorphy on $\Re(s_j)>0$ from \eqref{eq:Li_int'}. 
\end{proof}

\begin{rem}
H.~Tsumura pointed out to the author that the above lemma is 
a special case of Komori's result \cite{Ko}. 
\end{rem}

Now we define the main object of this article. 

\begin{defn}
For $\bs=(s_1,\ldots,s_r),\bs'=(s'_1,\ldots,s'_r)\in\C^r$ 
with $\Re(s_j)>0$, we define 
\begin{equation}\label{eq:eta_def}
\eta(\bs';\bs)\coloneqq 
\prod_{j=1}^r\frac{1}{\Gamma(s_j)}\idotsint_0^\infty 
\frac{\Li_{\bs'}(1-e^{t_1+\cdots+t_r},\ldots,1-e^{t_r})}
{(1-e^{t_1+\cdots+t_r})\cdots(1-e^{t_r})}
\prod_{j=1}^r t_j^{s_j-1}dt_j. 
\end{equation}
\end{defn}

Let us discuss the convergence of the integral \eqref{eq:eta_def} and its analytic continuation. 
For this purpose, the following estimate is useful: 

\begin{lem}\label{lem:bound}
The function $(e^{(u+t)/2})/(e^{u}+e^{t}-1)$ is bounded on 
\[u,t\in D_\eps\coloneqq\{z\in\C\mid \Re(z)\ge -\eps,\, -\eps\le\Im(z)\le\eps\} \]
for sufficiently small $\eps>0$. 
\end{lem}
\begin{proof}
Put $x=e^u-1/2$, $y=e^t-1/2$. We bound the product 
\[\biggl|\frac{e^{(u+t)/2}}{e^{u}+e^{t}-1}\biggr|=\frac{\Re(x)+\Re(y)}{\abs{x+y}}
\cdot\frac{\abs{x}+\abs{y}}{\Re(x)+\Re(y)}\cdot\frac{\abs{xy}^{\frac12}}{\abs{x}+\abs{y}}
\cdot\biggl|\frac{e^u}{x}\biggr|^{\frac12}\cdot\biggl|\frac{e^t}{y}\biggr|^{\frac12}\]
factorwise. 

The first and third factors are bounded by $1$ and $1/2$, respectively. 
It is also easy to see that the fourth and fifth factors are bounded; 
indeed, $e^u/(e^u-1/2)$ is a continuous function on $D_\eps$, 
hence is bounded on any bounded region, and tends to $1$ with $\Re(u)\to\infty$. 

To bound the second factor, note that there is a constant $\theta=\theta_\eps>0$, 
depending only on $\eps$, such that $\abs{\arg x},\abs{\arg y}\le\theta$ holds for any $u,t\in D_\eps$. 
If $\eps>0$ is sufficiently small, $\theta$ also becomes arbitrarily small. 
In particular, we may assume that $\theta<\pi/2$. 
Then we have $\Re(x)\ge\abs{x}\cos\theta$, $\Re(y)\ge\abs{y}\cos\theta$ and hence 
\[\frac{\abs{x}+\abs{y}}{\Re(x)+\Re(y)}\le \frac{1}{\cos\theta}, \]
which gives a bound for the second factor. 
\end{proof}

To examine the convergence of the integral \eqref{eq:eta_def}, 
we substitute the expression \eqref{eq:Li_int} into it. Then we have 
\begin{equation}\label{eq:eta substitute}
\eta(\bs';\bs)=\prod_{j=1}^r\frac{\Gamma(1-s'_j)}{2\pi ie^{\pi is'_j}\Gamma(s_j)}
\int_{(\R_{>0})^r}\int_{(C_\eps)^r} 
\prod_{j=1}^r\frac{u_j^{s'_j-1}t_j^{s^{}_j-1}du_j\,dt_j}{e^{u_j+\cdots+u_r}+e^{t_j+\cdots+t_r}-1}.
\end{equation}
By Lemma \ref{lem:bound} applied to $D_{r\eps}$ in place of $D_\eps$, 
this integrand is bounded by a constant multiple of 
\[\Biggl|\prod_{j=1}^r\frac{u_j^{s'_j-1}t_j^{s^{}_j-1}}{e^{(u_j+\cdots+u_r+t_j+\cdots+t_r)/2}}\Biggr|
=\prod_{j=1}^r\Biggl|\frac{u_j^{s'_j-1}}{e^{ju_j/2}}\Biggr|
\cdot\prod_{j=1}^r\Biggl|\frac{t_j^{s^{}_j-1}}{e^{jt_j/2}}\Biggr|. \]
Hence the absolute convergence of the integral \eqref{eq:eta_def} is reduced to 
those of one variable integrals, which are elementary and well-known.  
Note that, if $\Re(s_j), \Re(s'_j)>0$, we may replace \eqref{eq:eta substitute} 
by a simpler and symmetric expression 
\begin{equation}\label{eq:eta_int'}
\eta(\bs';\bs)=\prod_{j=1}^r\frac{1}{\Gamma(s_j)\Gamma(s'_j)}
\idotsint_0^\infty \prod_{j=1}^r 
\frac{u_j^{s'_j-1}t_j^{s_j-1}du_j\,dt_j}{e^{u_j+\cdots+u_r}+e^{t_j+\cdots+t_r}-1}. 
\end{equation}

Moreover, the same estimate shows that we can transform each integral on $t_j\in\R_{>0}$ 
in \eqref{eq:eta substitute} to the integral along the contour $C_\eps$ to obtain 
\begin{equation}\label{eq:eta_int}
\eta(\bs';\bs)=\prod_{j=1}^r\frac{\Gamma(1-s_j)\Gamma(1-s'_j)}
{(2\pi i)^2e^{\pi(s_j+s'_j)}}\int_{(C_\eps)^{2r}}
\prod_{j=1}^r\frac{u_j^{s'_j-1}t_j^{s^{}_j-1}du_j\,dt_j}
{e^{u_j+\cdots+u_r}+e^{t_j+\cdots+t_r}-1}, 
\end{equation}
and that this integral is convergent for any $\bs,\bs'\in\C^r$. 
Therefore, we have shown the following: 

\begin{prop}\label{prop:eta}
The function $\eta(\bs';\bs)$ can be holomorphically continued 
to $\C^r\times\C^r$, and satisfies $\eta(\bs';\bs)=\eta(\bs;\bs')$. 
\end{prop}

\begin{rem}
Recall that the Euler-Zagier multiple zeta function 
\[\zeta(s_1,\ldots,s_r)=\sum_{0<n_1<\cdots<n_r}
\frac{1}{n_1^{s_1}\cdots n_r^{s_r}}\]
has the integral expression 
\[\zeta(s_1,\ldots,s_r)=\prod_{j=1}^r\frac{1}{\Gamma(s_j)}
\idotsint_0^\infty \prod_{j=1}^r 
\frac{u_j^{s_j-1}du_j}{e^{u_j+\cdots+u_r}-1}. \]
This formula, together with \eqref{eq:eta_int'}, suggests 
that our function $\eta(\bs';\bs)$ may be regarded as 
a `double multiple zeta function.' 
Furthermore, we may also consider a `multiple multiple zeta function' 
\[\eta(\bs_1,\ldots,\bs_l)=
\prod_{i=1}^l\prod_{j=1}^r\frac{1}{\Gamma(s_{ij})}
\idotsint_0^\infty \prod_{j=1}^r 
\frac{\prod_{i=1}^l t_{ij}^{s_{ij}-1}dt_{ij}}
{\sum_{i=1}^l e^{t_{ij}+\cdots+t_{ir}}-1}\]
for $\bs_i=(s_{i1},\ldots,s_{ir})$. 
In this paper, however, we don't pursue such a generalization 
for $l\geq 3$. 
\end{rem}

\section{Special values at positive integers}
From now on, we study the values of $\eta$ at positive integers. 
In particular, we are interested in the relationship between 
these values and the multiple zeta values. 
First let us recall some basic notation on the multiple zeta values. 

A finite sequence $\bk=(k_1,\ldots,k_r)$ of positive integers 
is called an \emph{index}. We put 
\[\wt{\bk}\coloneqq k_1+\cdots+k_r,\quad d(\bk)\coloneqq r, \]
and call them the \emph{weight} and the \emph{depth} of $\bk$, 
respectively. 

An index $\bk=(k_1,\ldots,k_r)$ is called \emph{admissible} if $k_r>1$ 
(or $r=0$, that is, $\bk$ is the empty index). 
When this is the case, the multiple zeta value $\zeta(\bk)$ is defined by 
the multiple series \eqref{eq:MZV}, 
and has the iterated integral expression 
\begin{equation}\label{eq:MZV_int}
\zeta(\bk)=\int_{(x_{ji})\in \Delta(\bk)}
\prod_{j=1}^r\frac{dx_{j1}}{1-x_{j1}}
\frac{dx_{j2}}{x_{j2}}\cdots\frac{dx_{jk_j}}{x_{jk_j}}. 
\end{equation}
Here $\Delta(\bk)$ is a domain of dimension $\wt{\bk}$ defined by 
\[\Delta(\bk)\coloneqq
\Biggl\{(x_{ji})_{\substack{j=1,\ldots,r,\\i=1,\ldots,k_j}}\Biggm| 
\begin{array}{r}
0<x_{11}<\cdots<x_{1k_1}<x_{21}<\cdots<x_{2k_2}\\
\cdots<x_{r1}<\cdots<x_{rk_r}<1
\end{array}\Biggr\}. \]

Now we return to the study of $\eta$-values. 
We start with an integral expression for $\eta(\bk;\bl)$ 
similar to \eqref{eq:MZV_int}, using domains of the form 
\[\nabla(\bk)\coloneqq
\Biggl\{(x_{ji})_{\substack{j=1,\ldots,r,\\i=1,\ldots,k_j}}\Biggm| 
\begin{array}{r}
1>x_{11}>\cdots>x_{1k_1}>x_{21}>\cdots>x_{2k_2}\\
\cdots>x_{r1}>\cdots>x_{rk_r}>0
\end{array}\Biggr\}. \]

\begin{prop}
For any indices $\bk=(k_1,\ldots,k_r),\bl=(l_1,\ldots,l_r)$ 
of depth $r>0$, we have 
\begin{equation}\label{eq:eta(k;l)_int}
\begin{split}
\eta(\bk;\bl)
=\int_{\substack{(x_{ji})\in\nabla(\bk)\\ (y_{ji})\in\nabla(\bl)}}
\prod_{j=1}^r\Biggl\{
\frac{dx_{j1}\,dy_{j1}}{1-x_{j1}y_{j1}}
&\frac{dx_{j2}}{1-x_{j2}}\cdots
\frac{dx_{jk_j}}{1-x_{jk_j}}\\
\times&\frac{dy_{j2}}{1-y_{j2}}\cdots
\frac{dy_{jl_j}}{1-y_{jl_j}}\Biggr\}. 
\end{split}
\end{equation}
\end{prop}
\begin{proof}
Since all variables are positive, 
we can use the integral \eqref{eq:eta_int'}: 
\[\eta(\bk;\bl)=\prod_{j=1}^r\frac{1}{\Gamma(k_j)\Gamma(l_j)}
\idotsint_0^\infty 
\prod_{j=1}^r\frac{u_j^{k_j-1}t_j^{l_j-1}du_j\,dt_j}
{e^{u_j+\cdots+u_r}+e^{t_j+\cdots+t_r}-1}. \]
Then we make the change of variables 
\[x_j=1-e^{-(u_j+\cdots+u_r)},\quad y_j=1-e^{-(t_j+\cdots+t_r)}, \]
which leads to 
\begin{multline}\label{eq:eta(k;l)_int1}
\eta(\bk;\bl)=\prod_{j=1}^r\frac{1}{\Gamma(k_j)\Gamma(l_j)}
\int_{\substack{1>x_1>\cdots>x_r>x_{r+1}=0\\ 
1>y_1>\cdots>y_r>y_{r+1}=0}}\\
\prod_{j=1}^r\biggl(\log\frac{1-x_{j+1}}{1-x_j}\biggr)^{k_j-1}
\biggl(\log\frac{1-y_{j+1}}{1-y_j}\biggr)^{l_j-1}
\frac{dx_j\,dy_j}{1-x_jy_j}. 
\end{multline}
Moreover, we have 
\begin{align*}
\frac{1}{\Gamma(k_j)}
\biggl(\log\frac{1-x_{j+1}}{1-x_j}\biggr)^{k_j-1}
&=\frac{1}{(k_j-1)!}
\biggl(\int_{x_{j+1}}^{x_j}\frac{dx}{1-x}\biggr)^{k_j-1}\\
&=\int_{x_j>x_{j2}>\cdots>x_{jk_j}>x_{j+1}}
\frac{dx_{j2}}{1-x_{j2}}\cdots\frac{dx_{jk_j}}{1-x_{jk_j}}, 
\end{align*}
and a similar formula for $\frac{1}{\Gamma(l_j)}
\bigl(\log\frac{1-y_{j+1}}{1-y_j}\bigr)^{l_j-1}$. 
If we substitute them into \eqref{eq:eta(k;l)_int1}, 
we get the result \eqref{eq:eta(k;l)_int} 
(with $x_j=x_{j1}$ and $y_j=y_{j1}$). 
\end{proof}

\begin{cor}
For indices $\bk=(k_1,\ldots,k_r), \bl=(l_1,\ldots,l_r)$,  
we have 
\begin{equation}\label{eq:eta_ser}
\eta(\bk;\bl)=\sum_{m_i,n_i:(*)}
\prod_{i=1}^{\wt{\bk}}
\frac{1}{m_i+m_{i+1}+\cdots+m_{\wt{\bk}}}
\prod_{i=1}^{\wt{\bl}}
\frac{1}{n_i+n_{i+1}+\cdots+n_{\wt{\bl}}}, 
\end{equation}
where the summation is taken over positive integers 
$m_1,\ldots,m_{\wt{\bk}}$ and $n_1,\ldots,n_{\wt{\bl}}$ 
satisfying 
\begin{multline}
m_1=n_1,\ m_{k_1+1}=n_{l_1+1},\ 
m_{k_1+k_2+1}=n_{l_1+l_2+1}, \ldots\\
\ldots,m_{k_1+\cdots+k_{r-1}+1}=n_{l_1+\cdots+l_{r-1}+1}. 
\tag{$*$}
\end{multline}
\end{cor}
\begin{proof}
We expand all factors of the integrand of \eqref{eq:eta(k;l)_int} by 
\[\frac{1}{1-x_{j1}y_{j1}}=\sum_{m=1}^\infty x_{j1}^{m-1}y_{j1}^{m-1},\ 
\frac{1}{1-x_{ji}}=\sum_{m=1}^\infty x_{ji}^{m-1},\ 
\frac{1}{1-y_{ji}}=\sum_{n=1}^\infty y_{ji}^{n-1}. \]
Then, with a renumbering of variables, the integral becomes 
\[\int_{\substack{1>x_1>\cdots>x_{\wt{\bk}}>0\\ 
1>y_1>\cdots>y_{\wt{\bl}}>0}}
\sum_{m_i,n_i:(*)}
\prod_{i=1}^{\wt{\bk}}x_i^{m_i-1}dx_i 
\prod_{i=1}^{\wt{\bl}}y_i^{n_i-1}dy_i. \]
By exchanging the integral and the summation, 
and by integrating repeatedly, we obtain the formula \eqref{eq:eta_ser}. 
\end{proof}

Let $\bX=(X_1,\ldots,X_r)$ and $\bY=(Y_1,\ldots,Y_r)$ be 
$r$-tuples of indeterminates, and define 
the generating function for the values $\eta(\bk;\bl)$ by 
\begin{equation}\label{eq:F_r}
F_r(\bX;\bY)=\sum_{\bk,\bl\in(\Z_{>0})^r}
\eta(\bk;\bl)X_1^{k_1-1}\cdots X_r^{k_r-1}
Y_1^{l_1-1}\cdots Y_r^{l_r-1}. 
\end{equation}

\begin{prop}
We have 
\begin{equation}\label{eq:F_r_int}
F_r(\bX,\bY)
=\int_{\substack{1>x_1>\cdots>x_r>0\\ 1>y_1>\cdots>y_r>0}}
\prod_{j=1}^r(1-x_j)^{X_{j-1}-X_j}(1-y_j)^{Y_{j-1}-Y_j}
\frac{dx_j\,dy_j}{1-x_jy_j}, 
\end{equation}
where $X_0=Y_0=0$. 
\end{prop}
\begin{proof}
We substitute the integral expression \eqref{eq:eta(k;l)_int1} 
of $\eta(\bk;\bl)$ into the definition \eqref{eq:F_r} of $F_r(\bX;\bY)$, 
and take the summation over $\bk$ and $\bl$ using 
\begin{align*}
\sum_{k_j=1}^\infty\frac{1}{\Gamma(k_j)}
\biggl(\log\frac{1-x_{j+1}}{1-x_j}\biggr)^{k_j-1}X_j^{k_j-1}
&=\exp\biggl(X_j\log\frac{1-x_{j+1}}{1-x_j}\biggr)\\
&=\biggl(\frac{1-x_{j+1}}{1-x_j}\biggr)^{X_j}
\end{align*}
and similar for $y_j$. Then we obtain 
\[F_r(\bX,\bY)
=\int_{\substack{1>x_1>\cdots>x_r>x_{r+1}=0\\ 
1>y_1>\cdots>y_r>y_{r+1}=0}}
\prod_{j=1}^r\biggl(\frac{1-x_{j+1}}{1-x_j}\biggr)^{X_j}
\biggl(\frac{1-y_{j+1}}{1-y_j}\biggr)^{Y_j}
\frac{dx_j\,dy_j}{1-x_jy_j}, \]
which is equal to \eqref{eq:F_r_int}. 
\end{proof}

\section{Relationship with multiple zeta values}
In this section, we discuss the relation of the values $\eta(\bk;\bl)$ 
with the multiple zeta values $\zeta(\bk)$. 
The first thing to be noticed is that the former can be written in terms of the latter. 

\begin{thm}\label{thm:eta and zeta}
Let $\calY$ (resp.~$\calZ$) denote the $\Q$-linear subspaces of $\R$ 
spanned by $\eta(\bk;\bl)$ for all indices $\bk$ and $\bl$ of the same depth 
(resp. spanned by $\zeta(\bk)$ for all admissible indices $\bk$). 
Then we have $\calY\subset\calZ$. 
\end{thm}
\begin{proof}
Let us make the change of variables $y_{ji}\rightsquigarrow y_{ji}^{-1}$ 
in the integral \eqref{eq:eta(k;l)_int}. Then we have 
\begin{equation}\label{eq:eta(k;l)_SelbergInt}
\begin{split}
\eta(\bk;\bl)=\int_{\substack{(x_{ji})\in\nabla(\bk)\\ (y_{ji})\in\Delta'(\bl)}}
\prod_{j=1}^r\Biggl\{
\frac{dx_{j1}\,dy_{j1}}{(y_{j1}-x_{j1})y_{j1}}
\frac{dx_{j2}}{1-x_{j2}}\cdots
\frac{dx_{jk_j}}{1-x_{jk_j}}&\\
\times\frac{dy_{j2}}{(y_{j2}-1)y_{j2}}\cdots
\frac{dy_{jl_j}}{(y_{jl_j}-1)y_{jl_j}}&\Biggr\}, 
\end{split}
\end{equation}
where 
\[\Delta'(\bl)\coloneqq
\Biggl\{(y_{ji})_{\substack{j=1,\ldots,r,\\i=1,\ldots,l_j}}\Biggm| 
\begin{array}{r}
1<y_{11}<\cdots<y_{1l_1}<y_{21}<\cdots<y_{2l_2}\\
\cdots<y_{r1}<\cdots<y_{rl_r}
\end{array}\Biggr\}. \]
This is a period integral on the moduli space $\overline{\mathfrak{M}}_{0,n}$ 
of genus zero curves with $n$ marked points, where $n=\wt{\bk}+\wt{\bl}+3$. 
Hence we can apply the theorem of Brown \cite[Theorem 1.1]{Br} 
to deduce that $\eta(\bk;\bl)$ is a $\Q[2\pi i]$-linear combination of multiple zeta values. 
Since $\eta(\bk;\bl)$ is a real number, in fact, this is a $\Q[(2\pi i)^2]$-linear combination. 
Thus we have $\eta(\bk;\bl)\in\calZ$, since $(2\pi i)^2=-24\zeta(2)$ and 
$\calZ$ is a $\Q$-subalgebra of $\R$. 
\end{proof}

\begin{rem}\label{rem:Y subset Z}
In the first version of this paper, whether $\eta(\bk;\bl)\in\calZ$ or not 
was asked as an open question. 
Then E.~Panzer, who read it on the arXiv, immediately communicated to the author 
that it can be shown by using Brown's result as above. 
More precisely, he claimed that the inclusion $\calY_w\subset\calZ_w$ can be obtained 
for any integer $w\ge 0$, 
where $\calY_w$ (resp.\ $\calZ_w$) denotes the space spanned by $\eta(\bk;\bl)$ 
with $\wt{\bk}+\wt{\bl}=w$ (resp.\ spanned by $\zeta(\bk)$ with $\wt{\bk}=w$). 
However, the author could not confirm this statement 
by simply applying Brown's result as in the above proof. 

Recently, K.~Ito announced that he and Sato showed the inclusion 
$\calY_w\subset\calZ_w$ \cite[Remark 4.2]{Ito}. 

Brown's work \cite{Br} actually give an algorithm to express the value $\eta(\bk;\bl)$ 
as a linear combination of multiple zeta values for any given indices $\bk$ and $\bl$ 
(this was also noted by Panzer). 
Such an algorithm may be available also from the method of Ito and Sato. 
On the other hand, as far as the author knows, any explicit formula of such expression 
for general indices is still unknown. 
See \eqref{eq:r=1} and \eqref{eq:r=k=l} below for such formulas in two (very) special cases. 
\end{rem}

Next we show a `sum formula' for $\eta$-values. 
First let us introduce some notation on indices. 
\begin{defn}
\begin{enumerate}
\item For an index $\bk=(k_1,\ldots,k_r)$ of weight $k$, put 
\[J(\bk)\coloneqq\{k_1,k_1+k_2,\ldots,k_1+\cdots+k_{r-1}\}
\subset \{1,2,\ldots,k-1\}. \]
\item We say that $\bk$ is a \emph{refinement} of $\bk'$, 
and denote $\bk\succeq\bk'$, 
if $\wt{\bk}=\wt{\bk'}$ and $J(\bk)\supset J(\bk')$. 
\item For a formal linear combination 
$\alpha=\sum_{\bk}a_\bk \bk$ of (finitely many) 
admissible indices, we linearly extend the function $\zeta$, 
i.e., set $\zeta(\alpha)=\sum_{\bk}a_\bk\zeta(\bk)$. 
This `linear extension' principle also applies to operations below. 
\item For an index $\bk$, we denote by $\bk^\star$ 
the formal sum $\sum_{\bk'\preceq\bk}\bk'$ of 
all indices $\bk'$ of which $\bk$ is a refinement. 
\item For indices $\bk$ and $\bl$, we denote $\bk*\bl$ 
the \emph{harmonic product} of $\bk$ and $\bl$. 
It is a formal sum of indices defined inductively by 
\begin{align*}
\varnothing *\bk&=\bk*\varnothing=\bk,\\
(k_1,\ldots,k_r)*(l_1,\ldots,l_s)
&=\bigl((k_1,\ldots,k_{r-1})*(l_1,\ldots,l_s),k_r\bigr)\\
&+\bigl((k_1,\ldots,k_r)*(l_1,\ldots,l_{s-1}),l_s\bigr)\\
&+\bigl((k_1,\ldots,k_{r-1})*(l_1,\ldots,l_{s-1}),k_r+l_s\bigr), 
\end{align*}
where $\varnothing$ denotes the unique index of depth $0$. 
\item For indices $\bk=(k_1,\ldots,k_r)$ and $\bl=(l_1,\ldots,l_s)$ 
with $r,s>0$, we set 
\[\bk\circledast\bl\coloneqq 
\bigl((k_1,\ldots,k_{r-1})*(l_1,\ldots,l_{s-1}),k_r+l_s\bigr). \]
\end{enumerate}
\end{defn}

\begin{thm}\label{thm:sum formula}
Denote by $I(k,r)$ the set of all indices of weight $k$ and depth $r$. 
Then we have, for any positive integers $k$, $l$ and $r$, 
\begin{equation}\label{eq:sum formula}
\sum_{\bk\in I(k,r)}\sum_{\bl\in I(l,r)}\eta(\bk;\bl)
=\zeta\bigl((\underbrace{2,\ldots,2}_r)\circledast 
(\underbrace{1,\ldots,1}_{k-r},0)^\star\circledast 
(\underbrace{1,\ldots,1}_{l-r},0)^\star\bigr). 
\end{equation}
\end{thm}
\begin{proof}
First note that, while $(\underbrace{1,\ldots,1}_{k-r},0)$ and 
$(\underbrace{1,\ldots,1}_{l-r},0)$ in the right hand side contain $0$, 
the formal operations work well and produce a formal sum of 
admissible indices of weight $k+l$. 
Explicitly, it is given by the multiple series 
\begin{equation}\label{eq:sum formula_RHS}
\sum_{\substack{m_1>\cdots>m_r>0\\ 
0<a_1\leq\cdots\leq a_{k-r}\leq m_1\\ 
0<b_1\leq\cdots\leq b_{l-r}\leq m_1}}
\frac{1}{m_1^2\cdots m_r^2a_1\cdots a_{k-r}b_1\cdots b_{l-r}}. 
\end{equation}

Let us compute the left hand side of \eqref{eq:sum formula}. 
Since 
\[F_r(X,\ldots,X;Y,\ldots,Y)=\sum_{k,l\geq r}
\Biggl(\sum_{\bk\in I(k,r),\bl\in I(l,r)}\eta(\bk;\bl)\Biggr)X^{k-r}Y^{l-r},\]
it is sufficient to compute this generating function. 
By \eqref{eq:F_r_int}, we have 
\[F_r(X,\ldots,X;Y,\ldots,Y)
=\int_{\substack{1>x_1>\cdots>x_r>0\\ 1>y_1>\cdots>y_r>0}}
(1-x_1)^{-X}(1-y_1)^{-Y}
\prod_{j=1}^r\frac{dx_j\,dy_j}{1-x_jy_j}. \]
Using the expansion 
$\frac{1}{1-x_jy_j}=\sum_{n_j=1}^\infty (x_jy_j)^{n_j-1}$ 
and integrating with respect to $x_r,y_r,\ldots,x_2,y_2$, 
it can be rewritten as 
\[\sum_{m_1>\cdots>m_r>0}\frac{1}{m_2^2\cdots m_r^2}
\int_0^1(1-x)^{-X}x^{m_1-1}dx\int_0^1(1-y)^{-Y}y^{m_1-1}dy. \]
Then we note that 
\begin{align*}
\int_0^1(1-x)^{-X}x^{m-1}dx
&=B(1-X,m)=\frac{\Gamma(1-X)\Gamma(m)}
{\Gamma(1-X+m)}\\
&=\frac{(m-1)!}{(1-X)(2-X)\cdots(m-X)}\\
&=\frac{1}{m}\prod_{a=1}^m\biggl(1-\frac{X}{a}\biggr)^{-1}
=\frac{1}{m}\prod_{a=1}^m\sum_{n=0}^\infty \frac{X^n}{a^n}\\
&=\frac{1}{m}\sum_{k=0}^\infty X^k
\sum_{0<a_1\leq\cdots\leq a_k\leq m}
\frac{1}{a_1\cdots a_k} 
\end{align*}
to obtain 
\begin{align*}
&F_r(X,\ldots,X;Y,\ldots,Y)\\
&=\sum_{k,l\geq 0}X^{k}Y^{l}
\sum_{m_1>\cdots>m_r>0}\frac{1}{m_1^2\cdots m_r^2}
\sum_{\substack{0<a_1\leq\cdots\leq a_k\leq m_1\\ 
0<b_1\leq\cdots\leq b_l\leq m_1}}
\frac{1}{a_1\cdots a_k b_1\cdots b_l}. 
\end{align*}
Hence the coefficient of $X^{k-l}Y^{l-r}$ 
coincides with the series \eqref{eq:sum formula_RHS}, 
and the result follows. 
\end{proof}

\begin{ex}\label{ex:explicit examples}
\begin{enumerate}
\item When $r=1$, \eqref{eq:sum formula} says that 
\begin{equation}\label{eq:r=1}
\begin{split}
\eta(k;l)
&=\sum_{\substack{m>0\\
0<a_1\leq\cdots\leq a_{k-1}\leq m\\
0<b_1\leq\cdots\leq b_{l-1}\leq m}}
\frac{1}{m^2a_1\cdots a_{k-1} b_1\cdots b_{l-1}}\\
&=\zeta\bigl((\underbrace{1,\ldots,1}_k)^\star\circledast
(\underbrace{1,\ldots,1}_l)^\star\bigr)
\end{split}
\end{equation}
for any $k,l>0$. For example, 
\begin{align*}
\eta(1;1)&=\zeta(2),\\
\eta(2;1)&=\zeta(1,2)+\zeta(3),\\
\eta(3;1)&=\zeta(1,1,2)+\zeta(2,2)+\zeta(1,3)+\zeta(4),\\
\eta(2;2)&=2\zeta(1,1,2)+\zeta(2,2)+2\zeta(1,3)+\zeta(4). 
\end{align*}
\item When $r=k=l$, \eqref{eq:sum formula} says that 
\begin{equation}\label{eq:r=k=l}
\eta(\underbrace{1,\ldots,1}_r;\underbrace{1,\ldots,1}_r)
=\zeta(\underbrace{2,\ldots,2}_r). 
\end{equation}
\end{enumerate}
\end{ex}

\appendix 
\section{Values of $\eta(k_1,\ldots,k_r;l)$ for positive integers $k_1,\ldots,k_r,l$}
The formula \eqref{eq:r=1} for $\eta(k;l)$, 
which is a special case of \eqref{eq:sum formula}, 
can also be generalized to another direction, namely, 
a formula for values $\eta(k_1,\ldots,k_r;l)$ of 
the function \eqref{eq:eta(r;1)}. 
Such a formula was first found by M.~Kaneko as a conjecture. 
In this appendix, we prove it. 

To state the formula, we recall the notion of the Hoffman dual $\bk^\vee$ 
of an index $\bk$.  This is the unique index 
such that $\wt{\bk^\vee}=\wt{\bk}$ and 
$J(\bk^\vee)=\{1,\ldots,\wt{\bk}-1\}\setminus J(\bk)$. 

\begin{thm}\label{thm:(r,1) formula}
For a nonempty index $\bk=(k_1,\ldots,k_r)$ and 
an integer $l>0$, we have 
\begin{equation}\label{eq:(r,1) formula}
\eta(\bk;l)=(-1)^{d(\bk^\vee)}\sum_{\bk'\succeq\bk^\vee}
(-1)^{d(\bk')}\zeta\bigl((\bk')^\star\circledast
(\underbrace{1,\ldots,1}_l)^\star\bigr). 
\end{equation}
\end{thm}

When $r=1$, the Hoffman dual of $(k)$ is $(\underbrace{1,\ldots,1}_k)$, 
and it has no refinement other than itself. 
Hence we recover \eqref{eq:r=1} from \eqref{eq:(r,1) formula}. 

\smallskip 

To prove Theorem \ref{thm:(r,1) formula}, 
we need some preparations. 

Fix positive integers $k$ and $l$. 
In the following, indices denoted by $\bk$ or $\bk'$ are of weight $k$ 
and sets denoted by $J$ or $J'$ are subsets of $\{1,\ldots,k-1\}$. 
For such a set $J$, we put 
\[S_J\coloneqq\sum_{0<a_1\square_1a_2\square_2
\cdots\square_{k-1}a_k=b_l\geq\cdots\geq b_1>0}
\frac{1}{a_1\cdots a_kb_l\cdots b_1}, \]
where $\square_j$ for $j=1,\ldots,k-1$ denote 
the relational operators 
\[\square_j=\begin{cases}
< & (j\in J),\\
= & (j\notin J)
\end{cases} \]
(it also depends on $l$, which we fixed). 
It is easy to see, for $\bk=(k_1,\ldots,k_r)$, that 
\begin{equation}\label{eq:zeta_S}
S_{J(\bk)}
=\sum_{0<m_1<\cdots<m_r=b_l\geq\cdots\geq b_1>0}
\frac{1}{m_1^{k_1}\cdots m_r^{k_r}b_1\cdots b_l}
=\zeta\bigl(\bk\circledast(\underbrace{1,\ldots,1}_l)^\star\bigr). 
\end{equation}

\begin{lem}
For any index $\bk$ (of weight $k$), we have 
\begin{align}
\label{eq:xi_S}
\xi(\bk;l)&=S_{J(\bk)}, \\
\label{eq:zeta^star_S}
\zeta\bigl(\bk^\star\circledast(\underbrace{1,\ldots,1}_l)^\star\bigr)
&=\sum_{J\subset J(\bk)}S_J. 
\end{align}
\end{lem}
\begin{proof}
In the definition \eqref{eq:xi} of $\xi(\bk;s)$, 
make a change of variable $u=1-e^{-t}$ and substitute $s=l$. 
Then we have 
\begin{align*}
&\xi(\bk;l)\\
&=\frac{1}{(l-1)!}\int_0^1
\frac{\Li_{\bk}(u)}{u}\bigl(-\log(1-u)\bigr)^{l-1}du\\
&=\sum_{0<m_1<\cdots<m_r}\frac{1}{m_1^{k_1}\cdots m_r^{k_r}}
\int_0^1u^{m_r-1}du\int_{0<v_1<\cdots<v_{l-1}<u}
\frac{dv_1}{1-v_1}\cdots\frac{dv_{l-1}}{1-v_{l-1}}. 
\end{align*}
Here we use $-\log(1-u)=\int_0^u\frac{dv}{1-v}$. 
The identity \eqref{eq:xi_S} is obtained by 
computing the integration in the order $u,v_{l-1},\ldots,v_1$. 

The identity \eqref{eq:zeta^star_S} follows from the computation 
\begin{align*}
\zeta\bigl(\bk^\star\circledast(\underbrace{1,\ldots,1}_l)^\star\bigr)
&=\sum_{\bk'\preceq\bk}
\zeta\bigl(\bk'\circledast(\underbrace{1,\ldots,1}_l)^\star\bigr)\\
&=\sum_{\bk'\preceq\bk} S_{J(\bk')}
=\sum_{J\subset J(\bk)} S_J, 
\end{align*}
where we use \eqref{eq:zeta_S} and 
the correspondence between indices $\bk'\preceq\bk$ 
and sets $J\subset J(\bk)$. 
\end{proof}

In addition to the above lemma, 
we use the following identity 
due to Kaneko and Tsumura \cite[Proposition 3.2]{KT}: 
\begin{equation}\label{eq:eta_xi}
\eta(\bk;s)=(-1)^{d(\bk)-1}\sum_{\bk'\succeq\bk}\xi(\bk';s). 
\end{equation}

\begin{proof}[{Proof of Theorem \ref{thm:(r,1) formula}}]
Let us compute the sum in the right hand side of 
\eqref{eq:(r,1) formula}: 
\begin{align*}
\sum_{\bk'\succeq\bk^\vee}(-1)^{d(\bk')}
\zeta\bigl((\bk')^\star\circledast(\underbrace{1,\ldots,1}_l)\bigr)
&=\sum_{\bk'\succeq\bk^\vee}(-1)^{d(\bk')}
\sum_{J\subset J(\bk')}S_J\\
&=\sum_{J'\supset J(\bk^\vee)}(-1)^{\#J'+1}
\sum_{J\subset J'}S_J\\
&=\sum_{J}S_J\sum_{J'\supset J(\bk^\vee)\cup J}(-1)^{\#J'+1}. 
\end{align*}
Here, we use \eqref{eq:zeta^star_S} in the first step 
and the correspondence between indices 
$\bk'\succeq\bk^\vee$ and sets $J'\supset J(\bk^\vee)$ 
in the second step (note that $d(\bk')=\#J(\bk')+1$). 
The third step is just an exchange of summations. 

It is easily shown that 
\[\sum_{J'\supset J(\bk^\vee)\cup J}(-1)^{\#J'+1}
=\begin{cases}
0 & (J(\bk^\vee)\cup J\subsetneq\{1,\ldots,k-1\}), \\
(-1)^k & (J(\bk^\vee)\cup J=\{1,\ldots,k-1\}). 
\end{cases}\]
Because of the equivalence 
\[J(\bk^\vee)\cup J=\{1,\ldots,k-1\} 
\iff J\supset \{1,\ldots,k-1\}\setminus J(\bk^\vee)=J(\bk),\] 
we can continue the above computation as 
\begin{align*}
\sum_{J}S_J\sum_{J'\supset J(\bk^\vee)\cup J}(-1)^{\#J'+1}
&=(-1)^k \sum_{J\supset J(\bk)}S_J\\
&=(-1)^k \sum_{\bk'\succeq \bk}\xi(\bk';l)\\
&=(-1)^k(-1)^{d(\bk)-1}\eta(\bk;l), 
\end{align*}
using \eqref{eq:xi_S} and \eqref{eq:eta_xi}. 
Thus the desired identity \eqref{eq:(r,1) formula} is obtained 
if we notice that 
$k-1=\#J(\bk)+\#J(\bk^\vee)=(d(\bk)-1)+(d(\bk^\vee)-1)$, 
i.e., $k-(d(\bk)-1)=d(\bk^\vee)$. 
\end{proof}

\section*{Acknowledgments}
The author expresses his deep gratitude to Masanobu Kaneko 
and Hirofumi Tsumura for valuable discussions and comments. 
He also sincerely thanks Erik Panzer for communicating to him 
the proof of Theorem \ref{thm:eta and zeta}. 
Furthermore, he much appreciates the referee's careful reading 
and suggestion on the use of Lemma \ref{lem:bound} to clarify 
the proof of Proposition \ref{prop:eta}.

\end{document}